\documentclass[11pt]{article}
\usepackage{amsmath, amssymb, amsthm}
\usepackage{verbatim}
\usepackage{multicol}
\usepackage{hyperref}
\usepackage{enumerate}
\usepackage{tikz}

\date{}
\title{\vspace{-0.8cm}Saturation in random graphs }
\author{
D\'aniel Kor\'andi \thanks{Department of Mathematics, ETH, 8092 Zurich. Email: daniel.korandi@math.ethz.ch.}
\and
Benny Sudakov \thanks{Department of Mathematics, ETH, 8092 Zurich.
Email: benjamin.sudakov@math.ethz.ch. 
Research supported in part by SNSF grant 200021-149111.}
}

\theoremstyle{plain}
\newtheorem{THM}{Theorem}[section]
\newtheorem*{THM*}{Theorem}

\newtheorem{LEMMA}[THM]{Lemma}
\newtheorem{DEF}[THM]{Definition}

\newtheorem{CLAIM}[THM]{Claim}

\newtheorem*{OBS}{Observation}

\theoremstyle{definition}

\newcommand{\Prb}{\mathbf{P}}
\newcommand{\Exp}{\mathbf{E}}
\newcommand{\floor}[1]{\left\lfloor #1 \right\rfloor}

\newcommand{\subs}{\subseteq}
\newcommand{\eps}{\varepsilon}
\newcommand{\Bin}{Bin}

\newcommand{\mB}{\mathcal{B}}

\newcommand{\polylog}{\textrm{polylog }}
\newcommand{\sat}{sat}
\newcommand{\wsat}{w\textrm{-}sat}
\newcommand{\ex}{ex}
\newcommand{\loga}{\log_{\alpha}}
\newcommand{\logb}{\log_{\beta}}

\newcommand{\Grnp}{G^{r}(n,p)}
\newcommand{\Krs}{K^{r}_s}
\newcommand{\Krr}{K^{r}_{r+1}}
\newcommand{\Krt}{K^{r}_t}
\newcommand{\Krn}{K^{r}_n}

\begin{document}
\maketitle

\begin{abstract}
A graph $H$ is $K_s$-saturated if it is a maximal $K_s$-free graph, i.e., $H$ contains no clique on $s$ vertices, but the addition of any missing edge creates one. The minimum number of edges in a $K_s$-saturated graph was determined over 50 years ago by Zykov and independently by Erd\H{o}s, Hajnal and Moon. In this paper, we study the random analog of this problem: minimizing the number of edges in a maximal $K_s$-free subgraph of the Erd\H{o}s-R\'enyi random graph $G(n,p)$. We give asymptotically tight estimates on this minimum, and also provide exact bounds for the related notion of weak saturation in random graphs. Our results reveal some surprising behavior of these parameters.
\end{abstract}

\section{Introduction}

For some fixed graph $F$, a graph $H$ is said to be $F$-saturated if it is a maximal $F$-free graph, i.e., $H$ does not contain any copies of $F$ as a subgraph, but adding any missing edge to $H$ creates one. The saturation number $\sat(n,F)$ is defined to be the minimum number of edges in an $F$-saturated graph on $n$ vertices. Note that the maximum number of edges in an $F$-saturated graph is exactly the extremal number $\ex(n,F)$, so the saturation problem of finding $\sat(n,F)$ is in some sense the opposite of the Tur\'an problem.

The first results on saturation were published by Zykov \cite{Z49} in 1949 and independently by Erd\H{o}s, Hajnal and Moon \cite{EHM64} in 1964. They considered the problem for cliques and showed that $\sat(n,K_s)=(s-2)n-\binom{s-1}{2}$. Here the upper bound comes from the graph consisting of $s-2$ vertices connected to all other vertices. Since the 1960s, the saturation number $\sat(n,F)$ has been extensively studied for various different choices of $F$. For results in this direction, we refer the interested reader to the survey \cite{FFS}.

In the present paper we are interested in a different direction of extending the original problem, where saturation is restricted to some host graph other than $K_n$. For fixed graphs $F$ and $G$, we say that a subgraph $H\subs G$ is $F$-saturated in $G$ if $H$ is a maximal $F$-free subgraph of $G$. The minimum number of edges in an $F$-saturated graph in $G$ is denoted by $\sat(G,F)$. Note that with this new notation, $\sat(n,F)=\sat(K_n,F)$.

A question of this type already appeared in the above mentioned paper of Erd\H{o}s, Hajnal and Moon. They proposed the bipartite analog of the saturation problem, and even formulated a conjecture on the value of $\sat(K_{n,m},K_{s,s})$. This conjecture was independently verified by Wessel \cite{W66} and Bollob\'as \cite{B67}, while general  $K_{s,t}$-saturation in bipartite graphs was later studied in \cite{GKS15}. Several other host graphs have also been considered, including complete multipartite graphs \cite{FJPV, R} and hypercubes \cite{CG08,JP,MNS}.

\medskip
In recent decades, classic extremal questions of all kinds are being extended to random settings. Hence it is only natural to ask what happens with the saturation problem in random graphs. As usual, we let $G(n,p)$ denote the Erd\H{o}s-R\'enyi random graph on vertex set $[n]=\{1,\ldots,n\}$, where two vertices $i,j\in [n]$ are connected by an edge with probability $p$, independently of the other pairs. In this paper we study $K_s$-saturation in $G(n,p)$.

The corresponding Tur\'an problem of determining $\ex(G(n,p),K_s)$, the maximum number of edges in a $K_s$-saturated graph in $G(n,p)$, has attracted a considerable amount of attention in recent years. The first general results in this direction were given by Kohayakawa, R\"odl and Schacht \cite{KRS04}, and independently Szab\'o and Vu \cite{SV03} who proved a random analog of Tur\'an's theorem for large enough $p$. This
problem was resolved by Conlon and Gowers \cite{CG}, and independently by Schacht \cite{S}, who determined the correct range of edge probabilities where the Tur\'an-type theorem holds. The powerful method of hypergraph containers, developed by Balogh, Morris and Samotij \cite{BMS15} and by Saxton and Thomason \cite{ST15}, provides an alternative proof. Roughly speaking, these results establish that for most values of $p$, the random graph $G(n,p)$ behaves much like the complete graph as a host graph, in the sense that $K_s$-free subgraphs of maximum size are essentially $s-1$-partite.

\medskip
Now let us turn our attention to the saturation problem. When $s=3$, the minimum saturated graph in $K_n$ is the star. Of course we cannot exactly adapt this structure to the random graph, because the degrees in $G(n,p)$ are all close to $np$ with high probability, but we can do something very similar. Pick a vertex $v_1$ and include all its incident edges in $H$. This way, adding any edge of $G(n,p)$ induced by the neighborhood of $v_1$ creates a triangle, so we have immediately taken care of a $p^2$ fraction of the edges (which is the best we can hope to achieve with one vertex). Then the edges adjacent to some other vertex is expected to take care of a $p^2$ fraction of the remaining edges, and so on: we expect every new vertex to reduce the number of remaining edges by a factor of $(1-p^2)$.

Repeating this about $\log_{1/(1-p^2)} \binom{n}{2}$ times, we obtain a  $K_3$-saturated bipartite subgraph containing approximately $pn\log_{1/(1-p^2)} \binom{n}{2}$ edges. It feels natural to think that this construction is more or less optimal. Surprisingly, this intuition turns out to be incorrect. Indeed, we will present an asymptotically tight result which is better by a factor of $\frac{p\log(1-p)}{\log(1-p^2)}>1$. 
Moreover, rather unexpectedly, the asymptotics of the saturation numbers do not depend on $s$, the size of the clique.

\begin{THM} \label{thm:main}
Let $0<p<1$ be some constant probability and $s\ge 3$ be an integer. Then
\[ \sat(G(n,p),K_s)=(1+o(1))n\log_{\frac{1}{1-p}} n  \]
with high probability.
\end{THM}

\medskip

Our next result is about the closely related notion of weak saturation (also known as graph bootstrap percolation), introduced by Bollob\'as \cite{B68} in 1968. A graph $H\subs G$ is weakly $F$-saturated in $G$ if $H$ does not contain any copies of $F$, but the missing edges of $H$ in $G$ can be added back one-by-one in some order, such that every edge creates a new copy of $F$. The smallest number of edges in a weakly $F$-saturated graph in $G$ is denoted by $\wsat(G,F)$.

Clearly $\wsat(G,F)\le \sat(G,F)$, but Bollob\'as conjectured that when both $G$ and $F$ are complete, then in fact equality holds. This somewhat surprising fact was proved by Lov\'asz \cite{L77}, Frankl \cite{F82} and Kalai \cite{K84,K85} using linear algebra:

\begin{THM}[\cite{K84,K85}] \label{thm:weak_clique}
\[ \wsat(K_n, K_s) = \sat(K_n,K_s) = (s-2)n- \binom{s-1}{2} \]
\end{THM}

Again, many variants of this problem with different host graphs have been studied \cite{A85,MS,MNS}, and it turns out to be quite interesting for random graphs, as well. In this case we are able to determine the weak saturation number exactly. It is worth pointing out that this number is linear in $n$, as opposed to the saturation number, which is of the order $n\log n$.

\begin{THM} \label{thm:weaksat}
Let $0<p<1$ be some constant probability and $s\ge 3$ be an integer. Then
\[ \wsat(G(n,p),K_s) = (s-2)n- \binom{s-1}{2}  \]
with high probability.
\end{THM}

\medskip

We will prove Theorem~\ref{thm:main} in Section~\ref{sec:strongsat} and Theorem~\ref{thm:weaksat} in Section~\ref{sec:weaksat}. We finish the paper with a discussion of open problems in Section~\ref{sec:last}. 

\textbf{Notation.} All our results are about $n$ tending to infinity, so we often tacitly assume that $n$ is large enough. We say that some property holds {\em with high probability}, or {\em whp}, if the probability tends to 1 as $n$ tends to infinity. In this paper $\log$ stands for the natural logarithm unless specified otherwise in the subscript. For clarity of presentation, we omit floor and ceiling signs whenever they are not essential. We use the standard notations of $G[S]$ for the subgraph of $G$ induced by the vertex set $S$, and $G[S,T]$ for the (bipartite) subgraph of $G[S\cup T]$ containing the $S$-$T$ edges of $G$. For sets $A,B$ and element $x$, we will sometimes write $A+x$ for $A\cup \{x\}$ and $A-B$ for $A\setminus B$.

\section{Strong saturation} \label{sec:strongsat}

In this section we prove Theorem~\ref{thm:main} about $K_s$-saturation in random graphs.

Let us say that a graph $H$ completes a vertex pair $\{u,v\}$ if adding the edge $uv$ to $H$ creates a new copy of $K_s$. Using this terminology, a $K_s$-free subgraph $H\subs G$ is $K_s$-saturated in $G$, if and only if $H$ completes all edges of $G$ missing from $H$.

\medskip
We will make use of the following bounds on the tail of the binomial distribution.
\begin{CLAIM} \label{lem:chern}
Let $0<p<1$ be a constant and $X\sim \Bin (n,p)$ be a binomial random variable. Then, for sufficiently large $n$,
\begin{enumerate}
  \item $\Prb[X \ge np+a] \le e^{-\frac{a^2}{2(np+a/3)}}$,
  \item $\Prb[X \le np-a] \le e^{-\frac{a^2}{2np}}$ and
  \item $\Prb\left[X \le \frac{n}{\log^2 n}\right] \le (1-p)^{n-\frac{n}{\log n}}$.
\end{enumerate}
\end{CLAIM}
\begin{proof}
The first two statements are standard Chernoff-type bounds (see e.g. \cite{JLRBOOK}). The third can be proved using a straightforward union-bound argument as follows.

Let us think about $X$ as the cardinality of a random subset $A\subs [n]$, where every element in $[n]$ is included in $A$ with probability $p$, independently of the others. $X\le \frac{n}{\log^2 n}$ means that $A$ is a subset of some set $I\subs[n]$ of size $\frac{n}{\log^2 n}$. For a fixed $I$, the probability that $X\subs I$ is $(1-p)^{n-|I|}$. We can choose an $I$ of size $\frac{n}{\log^2 n}$ in
\[ \binom{n}{n/\log^2 n} \le \left(\frac{en}{n/\log^2 n}\right)^{\frac{n}{\log^2 n}} \le (e\log^2 n)^{\frac{n}{\log^2 n}} \le e^{\frac{3n\log\log n}{\log^2 n}}\]
different ways, so 
\[ \Prb\left[X \le \frac{n}{\log^2 n}\right] \le e^{\frac{3n\log\log n}{\log^2 n}} \cdot (1-p)^{n-\frac{n}{\log^2 n}} \le (1-p)^{n-\frac{n}{\log n}}. \]

\end{proof}

\subsection{Lower bound}

First, we prove that any $K_s$-saturated graph in $G(n,p)$ contains at least $(1+o(1))n\log_{1/(1-p)} n$ edges. In fact, our proof does not use the property that $K_s$-saturated graphs are $K_s$-free, only that adding any missing edge from the host graph creates a new copy of $K_s$. Now if such an edge creates a new $K_s$ then it will of course create a new $K_3$, as well, so it is enough to show that our lower bound holds for triangle-saturation.

\begin{THM} \label{thm:main_lower}
Let $0<p<1$ be a constant. Then with high probability, $G=G(n,p)$ satisfies the following. If $H$ is a subgraph of $G$ such that for any edge $e\in G$ missing from $H$, adding $e$ to $H$ creates a new triangle, then $H$ contains at least $n\log_{1/(1-p)} n- 6n\log_{1/(1-p)} \log_{1/(1-p)} n$ edges.
\end{THM}
\begin{proof}
Let $H$ be such a subgraph and set $\alpha=\frac{1}{1-p}$. Let $A$ be the set of vertices that are incident to at least $\loga^2 n$ edges in $H$, and let $B=[n]-A$ be the rest. If $|A|\ge \frac{2n}{\loga n}$ then $H$ contains at least $\frac{1}{2}|A|\loga^2n\ge n\loga n$ edges and we are done. So we may assume $|A|\le \frac{2n}{\loga n}$ and hence $|B|\ge n(1-\frac{2}{\loga n})$. Our aim is to show that whp every vertex in $B$ is adjacent to at least $\loga n- 5\loga \loga n$ vertices of $A$ in $H$. This would imply that $H$ contains at least $|B|(\loga n- 5\loga \loga n)\ge n(\loga n- 6\loga \loga n)$ edges, as needed.

So pick a vertex $v\in B$ and let $N$ be its neighborhood in $A$ in the graph $H$. Let $uv$ be an edge of the random graph $G$ missing from $H$. Since $H$ is $K_3$-saturated, we know that $u$ and $v$ must have a common neighbor $w$ in $H$. Notice that for all but at most $\loga^4 n$ choices of $u$, this $w$ must lie in $A$. Indeed, $v$ is in $B$, so there are at most $\loga^2 n$ choices for $w$ to be a neighbor of $v$, and if this $w$ is also in $B$, then there are only $\loga^2 n$ options for $u$, as well. So the neighbors of $N$ in $H\subs G$ must contain all but $\loga^4 n$ of the vertices (outside $N$) that are adjacent to $v$ in $G$. The following claim shows that this is only possible if $|N|\ge \loga n- 5\loga \loga n$, thus finishing our proof.
\end{proof}

\begin{CLAIM} \label{lem:lowerlem}
Let $0<p<1$ be a constant and $G=G(n,p)$. Then whp for any vertex $x$ and set $Q$ of size at most $\loga n - 5\loga \loga n$, there are at least $2\loga^4 n$ vertices in $G$ adjacent to $x$ but not adjacent to any of the vertices in $Q$.
\end{CLAIM}
\begin{proof}
Fix $x$ and $Q$. Then the probability that some other vertex $y$ is adjacent to $x$ but not to any of $Q$ is $p'=p(1-p)^{|Q|}\ge p\frac{\loga^5 n}{n}$. These events are independent for the different $y$'s, so the number of vertices satisfying this property is distributed as $\Bin(n-1-|Q|,p')$. Its expectation, $p'(n-1-|Q|)$ is at least $\frac{p}{2}\loga^5 n$, so the probability that there are fewer than $2\loga^4 n$ such vertices is, by Claim~\ref{lem:chern}, at most $e^{-\Omega(\log^5 n)}$. But there are only $n$ ways to choose $x$ and $\sum_{i=1}^{\loga n} \binom{n}{i}\le n\cdot n^{\loga n}$ ways to choose $Q$, so the probability that for some $x$ and $Q$ the claim fails is
\[ n^2\cdot n^{\loga n} \cdot e^{-\Omega(\log^5 n)} \le \exp( O(\log^2 n)- \Omega(\log^5 n)) =o(1).  \]
\end{proof}

\subsection{Upper bound}

Next, we construct a saturated subgraph of the random graph that contains $(1+o(1))n\log_{1/(1-p)} n$ edges. The following observation says that it is enough to find a graph that is saturated at almost all the edges.

\begin{OBS} \label{obs:almost}
It is enough to find a $K_s$-free graph $G_0$ that completes all but at most $o(n\log n)$ missing edges. A maximal $K_s$-free supergraph $G\supseteq G_0$ will then be $K_s$-saturated with an asymptotically equal number of edges.
\end{OBS}

Before we give a detailed proof of the upper bound, let us sketch the main ideas for the case $s=3$. For simplicity, we will also assume $p=\frac{1}{2}$.

\medskip
As we mentioned in the introduction, if we fix a set $A_1$ of $\log_{4/3}\binom{n}{2}\approx 2\log_{4/3} n$ vertices with $B_1=[n]-A_1$, then $G[A_1,B_1]$ is a $K_3$-saturated subgraph with about $n\log_{4/3} n$ edges. But we can do better than that (see Figure~\ref{fig:strongsat}):

So instead, we fix $A_1$ to be a set of $2\log_2 n$ vertices and add all edges in $G[A_1,B_1]$ to our construction. This way we complete most of the edges in $B_1$ using about $n\log_2 n$ edges. Of course, we still have plenty of edges in $G[B_1]$ left incomplete, however, as we shall see, almost all of them are induced by a small set $B_2\subs B_1$ of size $o(n)$. But then we can complete all the edges in $B_2$ using only $o(n\log n)$ extra edges: just take an additional set $A_2$ of $\log_{4/3} \binom{n}{2}$ vertices, and add all edges in $G[A_2,B_2]$ to our construction.

This way, however, we still need to take care of the $\Theta(n\log n)$ incomplete edges between $A_2$ and $B_3=B_1-B_2$. As a side remark, let us point out that dropping the $K_3$-freeness condition from $K_3$-saturation would make our life easier here. Indeed, then we could have just chosen $A_2$ to be a subset of $B_2$.

The trick is to take yet another set $A_3$ of $o(\log n)$ vertices, and add all the $o(n\log n)$ edges of $G$ between $A_3$ and $A_2\cup B_3$ to our construction. Now the $o(\log n)$ vertices in $A_3$ are not enough to complete all the edges between $A_2$ and $B_3$, but if $|A_3|=\omega(1)$, then it will complete {\em most} of them. This gives us a triangle-free construction on $(1+o(1))2\log_2 n$ edges completing all but $o(n\log n)$ edges, and by the observation above this is enough.

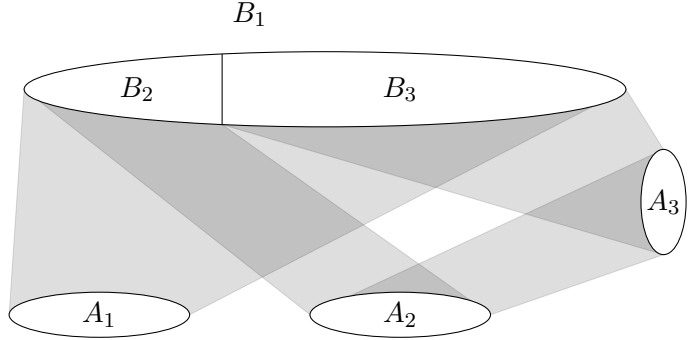
\begin{figure}[ht]
\begin{center}
\begin{tikzpicture}
\draw[fill, opacity=.13] (-1.8,-3)--(4,0)--(-4,0)--(-4.2,-3)--(-1.8,-3);
\draw[fill, opacity=.13] (2.2,-3)--(250:4 and .5)--(-4,0)--(-.2,-3)--(2.2,-3);
\draw[fill, opacity=.13] (4.5,-.8)--(4,0)--(250:4 and .5)--(4.5,-2.2)--(4.5,-.8);
\draw[fill, opacity=.13] (4.5,-.8)--(-.2,-3)--(2.2,-3)--(4.5,-2.2)--(4.5,-.8);

\draw[fill=white] (0,0) ellipse (4 and .5);
\node at (-1,1) {$B_1$};
\node at (-2.5,0) {$B_2$};
\node at (1,0) {$B_3$};
\draw[fill=white] (-3,-3) ellipse (1.2 and .3);
\node at (-3,-3) {$A_1$};
\draw[fill=white] (1,-3) ellipse (1.2 and .3);
\node at (1,-3) {$A_2$};
\draw[fill=white] (4.5,-1.5) ellipse (.3 and .7);
\node at (4.5,-1.5) {$A_3$};

\draw (110:4 and .5) -- (250:4 and .5);

\end{tikzpicture}
\caption{strong $K_3$-saturation}  \label{fig:strongsat}
\end{center}
\end{figure}
\medskip

Let us now collect the ingredients of the proof. The next lemma says that the graph comprised of the edges incident to a fixed set of $a$ vertices will complete all but an approximately $(1-p^2)^a$ fraction of the edges in any prescribed set $E$.

\begin{LEMMA} \label{lem:edgecover}
Let $s\ge 3$ be a fixed integer, and suppose $a=a(n)$ grows to infinity as $n\rightarrow\infty$. Let $A\subs [n]$ be a set of size $a$ and $E$ be a collection of vertex pairs from $B=[n]-A$. Now consider the random graph $G_A$ defined on $[n]$ as follows:
\begin{enumerate}
 \item $G_A[B]$ is empty,
 \item $G_A[A]$ is a fixed graph such that any induced subgraph on $\frac{a}{\log^2 a}$ vertices contains a copy of $K_{s-2}$,
 \item the edges between $A$ and $B$ are in $G_A$ independently with probability $p$.
\end{enumerate}
Then the expected number of pairs in $E$ that $G_A$ does not complete is at most $(1-p^2)^{a-\frac{a}{\log a}}\cdot |E|$, provided $n$ is sufficiently large.
\end{LEMMA}
\begin{proof}
Suppose a pair $\{u,v\}\in E$ of vertices is incomplete, i.e., adding the edge $uv$ to $G_A$ does not create a $K_s$. Because of the second condition, the probability of this event can be bounded from above by the probability that $u$ and $v$ have fewer than $\frac{a}{\log^2 a}$ neighbors in $A$. As the size of the common neighborhood of $u$ and $v$ is distributed as $\Bin(a,p^2)$, Lemma~\ref{lem:chern} implies that this probability is at most $(1-p^2)^{a-\frac{a}{\log a}}$. This bound holds for any pair in $E$, so the expected number of bad pairs is indeed no greater than $(1-p^2)^{a-\frac{a}{\log a}}\cdot |E|$. 
\end{proof}

A $K_s$-saturated graph cannot contain any cliques of size $s$. So if we want to construct such a graph using Lemma~\ref{lem:edgecover}, we need to make sure that $G_A$ itself is $K_s$-free. The easiest way to do this is by requiring that $G_A[A]$ not contain any $K_{s-1}$. In our application, $G_A$ will be a subgraph of $G(n,p)$, so in particular, we need $G_A[A]$ to be a subgraph of an Erd\H{o}s-R\'enyi random graph that is $K_{s-1}$-free but induces no large $K_{s-2}$-free subgraph. Krivelevich \cite{K95} showed that whp we can find such a subgraph.

\begin{THM}[Krivelevich] \label{thm:kriv}
Let $s\ge 3$ be an integer, $p\ge cn^{-2/s}$ (for some small constant $c$ depending on $s$) and $G=G(n,p)$. Then whp $G$ contains a $K_{s-1}$-free subgraph $H$ that has no $K_{s-2}$-free induced subgraph on $n^{\frac{s-2}{s}}\polylog n\le \frac{n}{\log^3 n}$ vertices.
\end{THM}

%In fact, Krivelevich proved that for $p=cn^{-2/(r+1)}$, deleting the edge set of a maximal collection of edge-disjoint $K_r$'s in $G=G(n,p)$ provides such a graph whp.
%
We will also use the fact that random graphs typically have relatively small chromatic numbers (see e.g. \cite{JLRBOOK}):

\begin{CLAIM} \label{lem:chrom}
Let $0<p<1$ be a constant and $G=G(n,p)$. Then $\chi(G)=(1+o(1))\frac{n}{2\log_{1/(1-p)}n}$ whp.
\end{CLAIM}
We are now ready to prove our main theorem.

\begin{THM}
Let $0<p<1$ be a constant, $s\ge 3$ be a fixed integer and $G=G(n,p)$. Then whp $G$ contains a $K_s$-saturated subgraph on $(1+o(1))n\log_{1/(1-p)}n$ edges.
\end{THM}
\begin{proof}
Define $\alpha=\frac{1}{1-p}$, $\beta=\frac{1}{1-p^2}$, and set $a_1=\frac{1}{p}(1+ \frac{3}{\log\loga n})\loga n $, $a_2=(1 + \frac{2}{\log\logb n})\logb \binom{n}{2}$ and $a_3=\frac{a_2}{\sqrt{\log a_2}}=o(\log n)$. Let us also define $A_1,A_2,A_3$ to be some disjoint subsets of $[n]$ such that $|A_i|=a_i$, and let $B_1=[n]-(A_1\cup A_2\cup A_3)$.

Expose the edges of $G[A_1]$. By Theorem~\ref{thm:kriv} it contains a $K_{s-1}$-free subgraph $G_1$ with no $K_{s-2}$-free subset of size $\frac{a_1}{\log^3 a_1}$ whp. Let $H_1$ be the graph on vertex set $A_1\cup B_1$ such that $H_1[A_1]=G_1$, $H_1[B_1]$ is empty, and $H_1$ is identical to $G$ between $A_1$ and $B_1$. Now define $B_2$ to be the set of vertices in $B_1$ that are adjacent to fewer than 
$(1+\frac{2}{\log\loga n})\loga n$ vertices of $A_1$ in the graph $H_1$. We claim that $H_1$ completes all but $O(n\log\log n)$ of the vertex pairs in $B_1$ {\em not} induced by $B_2$.

To see this, let $F$ be the set of incomplete pairs in $B_1$ not induced by $B_2$. Now fix a vertex $v\in B_1-B_2$ and denote its neighborhood in $A_1$ by $N_v$. By definition, $|N_v|\ge (1+\frac{2}{\log\loga n})\loga n$. Let $d_F(v)$ be the degree of $v$ in $F$, i.e., the number of pairs $\{u,v\}\subs B_1$ that $H_1$ does not complete. Conditioned on $N_v$, the size of the common neighborhood of $v$ and some $u\in B_1$ is distributed as $\Bin(|N_v|,p)$, so by Claim~\ref{lem:chern} the probability that this is smaller than $\frac{a_1}{\log^3 a_1}\le \frac{|N_v|}{\log^2|N_v|}$ is at most $(1-p)^{|N_v|-|N_v|/\log|N_v|} \le (1-p)^{\loga n}= \frac{1}{n}$. On the other hand, if the common neighborhood contains at least $\frac{a_1}{\log^3 a_1}$ vertices in $A_1$, then it also induces a $K_{s-2}$, thus the pair is complete. Therefore $\Exp[d_F(v)]\le 1$ and $\Exp[\sum_{v\in B_2-B_1} d_F(v)]\le n$.

Here the quantity $\sum_{v\in B_2-B_1} d_F(v)$ bounds $|F|$, the number of incomplete edges in $B_1$ that are not induced by $B_2$, so by Markov's inequality this number is indeed $O(n\log\log n)$ whp. By the Observation above, we can temporarily ignore these edges, and concentrate instead on completing those induced by $B_2$. To complete them, we will use the (yet unexposed) edges touching $A_2$.

\medskip
The good thing about $B_2$ is that it is quite small: For a fixed $v\in B_1$, the size of its neighborhood in $A_1$ is distributed as $\Bin(a_1,p)$, so the probability that it is smaller than $(1+\frac{2}{\log\loga n})\loga n =pa_1-\frac{\loga n}{\log\loga n}$ is, by Claim~\ref{lem:chern}, at most $e^{-\loga n/ 4 (\log\loga n)^2} \ll \frac{1}{\log n}$. So by Markov, $|B_2|$, the number of such vertices, is smaller than $\frac{n}{\log n}$ whp.

Now expose the edges of $G[A_2]$. Once again, whp we can apply Theorem~\ref{thm:kriv} to find a $K_{s-1}$-free subgraph $G_2$ that has no large $K_{s-2}$-free induced subgraph. Let $H_2$ be the graph on $A_2\cup B_2$ such that $H_2[A_2]=G_2$, $H_2[B_2]$ is empty, and $H_2=G$ on the edges between $A_2$ and $B_2$. Let us apply Lemma~\ref{lem:edgecover} to $G_{A_2}=H_2$ with $E$ containing all vertex pairs in $B_2$. The lemma says that the expected number of incomplete pairs in $E$ is at most
\[  (1-p^2)^{a_2-\frac{a_2}{\log a_2}} \cdot \binom{|B_2|}{2}\le (1-p^2)^{\logb \binom{n}{2}}\cdot  \binom{n/\log n}{2}=o(1), \]
so by Markov, $H_1$ completes all of $E$ whp, in particular it completes all the edges in $G[B_2]$. Note that $|B_2| \le \frac{n}{\log n}$ also implies that $H_2$ only contains $O(n)$ edges, so adding them to $H_1$ does not affect the asymptotic amount of edges in our construction.

However, the edges connecting $A_2$ to $B_3=B_1-B_2$ are still incomplete, and there are $\Theta(n\log n)$ of them, too many to ignore using the Observation. The idea is to complete most of these using the edges between $A_3$ and $A_2\cup B_3$, but we need to be a little bit careful not to create copies of $K_s$ with the edges in $H_2[A_2]=G_2$. We achieve this by splitting up $A_2$ into independent sets.

\medskip
By Claim~\ref{lem:chrom}, we know that $k=\chi(G[A_2])=O\left(\frac{a_2}{\log a_2}\right)$, so $G_2\subs G[A_2]$ can also be $k$-colored whp. Let $A_2=\cup_{i=1}^k A_{2,i}$ be a partitioning into color classes, so here $G_2[A_{2,i}]$ is an empty graph for every $i$. Let us also split $A_3$ into $2k$ parts of size $a_4=a_3/2k$ and expose the edges in $G[A_3]$. Each of the $2k$ parts contains a $K_{s-1}$-free subgraph with no $K_{s-2}$-free subset of size $\frac{a_4}{\log^3 a_4}$ with the same probability $p_0$, and by Theorem~\ref{thm:kriv}, $p_0=1-o(1)$ (note that $a_4=\Omega(\sqrt{\log a_2})$ grows to infinity). This means that the expected number of parts not having such a subgraph is $2k(1-p_0)=o(k)$, so by Markov's inequality, whp $k$ of the parts, $A_{3,1},\ldots A_{3,k}$, do contain such subgraphs $G_{3,1},\ldots,G_{3,k}$.

Now define $H_{3,i}$ to be the graph with vertex set $A_{3,i}\cup A_{2,i}\cup B_3$ such that $H_{3,i}[A_{3,i}]=G_{3,i}$, $H_{3,i}[A_{2,i}\cup B_3]$ is empty, and the edges between $A_{3,i}$ and $A_{2,i}\cup B_3$ are the same as in $G$. Then we can apply Lemma~\ref{lem:edgecover} to show that $H_{3,i}$ is expected to complete all but a $(1-p^2)^{a_4-a_4/\log a_4}$ fraction of the $A_{2,i}$-$B_3$ pairs. This means that the expected number of edges between $A_2$ and $B_3$ that $H_3=\cup_{i=1}^k H_{3,i}$ does not complete is bounded by 
\[ (1-p^2)^{\Omega(\sqrt{\log a_2})} \cdot |A_2||B_3| \le (1-p^2)^{\Omega(\sqrt{\log\log n})}\cdot O(n\log n).\]
So if $F'$ is the set of incomplete $A_2$-$B_3$ edges, then another application of Markov gives $|F'|=o(n\log n)$ whp.

\medskip
It is easy to check that $H=H_1\cup H_2\cup H_3$ is $K_s$-free, since it can be obtained by repeatedly ``gluing'' together $K_s$-free graphs along independent sets, and such a process can never create an $s$-clique. To estimate the size of $H$, note that Claim~\ref{lem:chern} implies that whp all the vertices in $A_1$ have $(1+o(1))p|B_1|$ neighbors in $B_1$, so $H_1$ contains $(1+o(1))n\loga n$ edges. We have noted above that $H_2$ contains $O(n)$ edges and $H_3$ clearly contains at most $|A_2\cup B_3||A_3|=O\left(n\frac{\log n}{\sqrt{\log\log n}}\right)$ edges. So in total, $H$ contains $(1+o(1))n\loga n$ edges.

Finally, the only edges $H$ is not saturated at are either in $F$, in $F'$, touching $A_3$, or induced by $A_1\cup A_2$. There are $o(n\log n)$ edges of each of these four kinds, so any maximal $K_s$-free supergraph $H'$ of $H$ has $(1+o(1))n\loga n$ edges. This $H'$ is $K_s$-saturated, hence the proof is complete.
\end{proof}

\section{Weak saturation} \label{sec:weaksat}

In this section we prove Theorem~\ref{thm:weaksat} about the weak saturation number of random graphs of constant density. In fact, we prove the statement for a slightly more general class of graphs, satisfying certain pseudorandomness conditions. We will need some definitions to formulate this result.

Given a graph $G$ and a vertex set $X$, a {\em clique extension} of $X$ is a vertex set $Y$ disjoint from $X$ such that $G[Y]$ is a clique and $G[X,Y]$ is a complete bipartite graph. We define the size of this extension to be $|Y|$.

The following definition of goodness captures the important properties needed in our proof.

\begin{DEF}
A graph $G$ is $(t,\gamma)$-good if $G$ satisfies the following properties:
\begin{enumerate}
 \item[P1.] For any vertex set $X$ of size $x$ and any integer $y$ such that $x,y\le t$, $G$ contains at least $\gamma n$ disjoint clique extensions of $X$ of size $y$.
 \item[P2.] For any two disjoint sets $S$ and $T$ of size at least $\gamma n/2$, there is a vertex $v\in T$ and $X\subs S$ of size $t-1$ such that $X\cup v$ induces a clique in $G$.
\end{enumerate}
\end{DEF} 

It is not hard to see that the Erd\H{o}s-R\'enyi random graphs satisfy the properties:

\begin{CLAIM} \label{lem:weak_conc}
Let $0<p<1$ be a constant and let $t$ be a fixed integer. Then there is a constant $\gamma=\gamma(p,t)>0$ such that whp $G=G(n,p)$ is $(t,\gamma)$-good.
\end{CLAIM}
\begin{proof}
To prove property P1, fix $x,y\le t$ and a set $X$ of size $x$, and split $V-X$ into groups of $y$ elements (with some leftover): $V_1,\ldots,V_m$ where $m=\floor{\frac{n-x}{y}}$. The probability that for some $i\in [m]$, all the pairs induced by $V_i$ or connecting $X$ and $V_i$ are edges in $G$ is $\tilde{p}=p^{\binom{y+x}{2}-\binom{x}{2}}$, which is a constant. Let $\mB_{X,y}$ be the event that fewer than $m\tilde{p}/2$ of the $V_i$ satisfy this. By Claim~\ref{lem:chern}, $\Prb(\mB_{X,y})\le e^{-m\tilde{p}/8}$.

Now set $\gamma=p^{\binom{2t}{2}}/4t$, then $\gamma n\le m\tilde{p}/2$, so if $\mB_{X,y}$ does not hold then there are at least $\gamma n$ different $V_i$'s that we can choose to be the sets $Y_i$ we are looking for. On the other hand, there are only $\sum_{x=0}^t\binom{n}{x}\le t\binom{n}{t} \le t\cdot e^{t\log n}$ choices for $X$ and $t$ choices for $y$, so
\[ \Prb(\cup_{X,y} \mB_{X,y}) \le t^2 \cdot e^{t\log n}\cdot e^{-\gamma n/4} = o(1),\]
hence P1 is satisfied whp.

\medskip
To prove property P2, notice that Theorem~\ref{thm:kriv} implies that whp any induced subgraph of $G$ on at least $\frac{n}{\log^3 n}$ vertices contains a clique of size $t$.\footnote{We should point out that this statement is much weaker than Theorem~\ref{thm:kriv} and can be easily proved directly using a simple union bound argument.} Let us assume this is the case and fix $S$ and $T$. Then if a vertex $v\in T$ has at least $\frac{n}{\log^3 n}$ neighbors in $S$, then the neighborhood contains a $t-1$-clique in $G$ that we can choose to be $X$. So if property P2 fails, then no such $v$ can have $\frac{n}{\log^3 n}$ neighbors in $S$.

The neighborhood of $v$ in $S$ is distributed as $\Bin(|S|,p)$, so the probability that $v$ has fewer than $\frac{n}{\log^3 n}\le \frac{|S|}{\log^2 |S|}$ neighbors in $S$ is at most $(1-p)^{|S|-|S|/\log|S|}\le e^{-p|S|/2}\le e^{-p\gamma n/4}$ by Claim~\ref{lem:chern}. These events are independent for the different vertices in $S$, so the probability that P2 fails for this particular choice of $S$ and $T$ is $e^{-\Omega(n^2)}$. But we can only fix $S$ and $T$ in $2^{2n}$ different ways, so whp P2 holds for $G$.
\end{proof}

Now we are ready to prove the following result, which immediately implies Theorem~\ref{thm:weaksat}.

\begin{THM} \label{thm:weak_gen}
Let $G$ be $(2s,\gamma)$-good. Then
\[ \wsat(G,K_s)=  (s-2)n- \binom{s-1}{2}. \]
\end{THM}
\begin{proof}
To prove the lower bound on $\wsat(G,K_s)$, it is enough to show that $G$ is weakly saturated in $K_n$. Indeed, if $H$ is weakly saturated in $G$ and $G$ is weakly saturated in $K_n$, then $H$ is weakly saturated in $K_n$: We can just add edges one-by-one to $H$, obtaining $G$, and then keep adding edges until we reach $K_n$ in such a way that every added edge creates a new copy of $K_s$. But then Theorem~\ref{thm:weak_clique} implies that $H$ contains at least $(s-2)n-\binom{s-1}{2}$ edges, which is what we want.

Actually, $G$ is not only weakly, but strongly saturated in $K_n$: property P1 implies that for any vertex pair $X=\{u,v\}$, $G$ contains a clique extension of $X$ of size $s-2$. But then adding the edge $uv$ to this subgraph creates the copy of $K_s$ that we were looking for.

\medskip

Let us now look at the upper bound on $\wsat(G,K_s)$.
Fix a set $C$ of $s-2$ vertices that induces a complete graph (such a $C$ exists by property P1, as it is merely a clique extension of $\emptyset$ of size $s-2$). Our saturated graph $H$ will consist of $G[C]$, plus $s-2$ edges for each vertex $v\in V-C$, giving a total of $\binom{s-2}{2}+(s-2)(n-s+2) =(s-2)n- \binom{s-1}{2}$ edges. We build $H$ in steps.

Let $V'$ be the set of vertices $v\in V-C$ adjacent to all of $C$. %Note that $|V'|\ge \gamma n$ by property P1. 
We start our construction with the graph $H_0\subs G$ on vertex set $V_0=C\cup V'$ that contains all the edges touching $C$.

Once we have defined $H_{i-1}$ on $V_{i-1}$, we pick an arbitrary vertex $v_i\in V-V_{i-1}$, and choose a set $C_i\subs V_{i-1}$ of $s-2$ vertices that induces a clique in $G$ with $v_i$. Again, we can find such a $C_i$ as a clique extension of $v_i\cup C$. By the definition of $V'$, this $C_i$ will lie in $V'$. Then we set $V_i=V_{i-1}\cup v_i$ and define $H_i$ to be the graph on $V_i$ that is the union of $H_{i-1}$ and the $s-2$ edges connecting $v_i$ to $C_i$. Repeating this, we eventually end up with some graph $H=H_l$ on $V=V_l$ that has $\binom{n}{2}-\binom{n-s+2}{2}$ edges. We claim it is saturated.

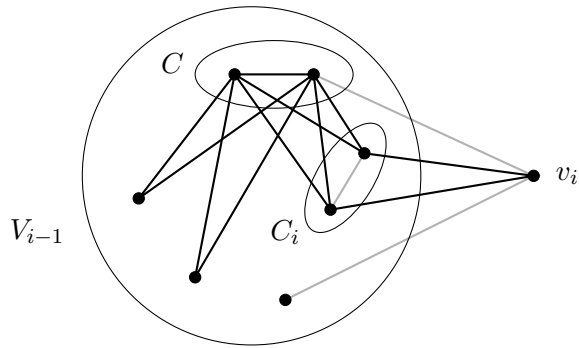
\begin{figure}[ht]
\begin{center}
\begin{tikzpicture}[scale=1.5]
\draw[black!30!white, thick] (1,0.7)--(0.7,0.2);
\draw[black!30!white, thick] (.55,1.4)--(2.5,0.5)--(0.3,-0.6);

\draw (.2,1.4) ellipse (.7 and .3);
\node at (-.7,1.5) {$C$};
\draw (0,0.5) ellipse (1.5 and 1.5);

\draw[fill] (-.15,1.4) circle (.05);
\draw[fill] (.55,1.4) circle (.05);
\draw[thick] (-.15,1.4)--(.55,1.4);

\foreach \p in {(-1,0.3),(-0.5,-0.4),(1,0.7),(0.3,-0.6),(0.7,0.2)}
  \draw[fill] \p circle (.05);

\foreach \p in {(-1,0.3),(-0.5,-0.4),(1,0.7),(0.7,0.2)}
  \draw[thick] (-.15,1.4)--\p--(.55,1.4);

\draw[fill] (2.5,0.5) circle (.05); 
\node at (2.8,0.5) {$v_i$};

\draw[thick] (1,0.7)--(2.5,0.5)--(0.7,0.2);
\draw[rotate=58] (.85,-0.45) ellipse (.55 and .25);
\node at (.3,0) {$C_i$};
\node at (-1.9,0) {$V_{i-1}$};

\end{tikzpicture}
\caption{Weak $K_4$-saturation. Black edges are in $H$, gray edges are in $G$ but not in $H$.}  \label{fig:weaksat}
\end{center}
\end{figure}

\medskip
We really prove a bit more: we show, by induction, that $H_i$ is weakly $K_s$-saturated in $G[V_i]$ for every $i$.
This is clearly true for $i=0$: any edge of $G[V_0]$ not contained in $H_0$ is induced by $V'$, and forms an $s$-clique with $C$.

Now assume the statement holds for $i-1$. We want to show that we can add all the remaining edges in $G_i=G[V_i]$ to $H_i$ one-by-one, each time creating a $K_s$. By induction, we can add all the edges in $G_{i-1}$, so we may assume they are already there. Then the only missing edges are the ones touching $v_i$.

If $v\in V_i$ is a clique extension of $C_i\cup v_i$ then adding $vv_i$ creates a $K_s$, so we can add this edge to the graph. Property P1 applied to $C\cup C_i\cup v_i$ to find clique extensions of size 1 shows that there are at least $\gamma n$ such vertices in $V'\subs V_i$. Let $N$ be the new neighborhood of $v_i$ after these additions, then $|N|\ge \gamma n$.

Now an edge $vv_i$ can also be added if $v \cup C'$ induces a complete subgraph in $G$, where $C'$ is any $s-2$-clique in $G_i[N]$. Let us repeatedly add the available edges (updating $N$ with every addition). We claim that all the missing edges will be added eventually.

Suppose not, i.e., some of the edges in $G_i$ touching $v_i$ cannot be added using this procedure. There cannot be more than $\gamma n/2$ of them, as that would contradict property P2 with $S=N$ and $T$ being the remaining neighbors of $v_i$ in $V_i$. But then take one such edge, $vv_i$, and apply property P1 to $C\cup \{v, v_i\}$. It shows that there are $\gamma n$ disjoint clique extensions of size $s-2$, that is, $\gamma n$ disjoint $s-2$-sets in $V'$ that form $s$-cliques with $\{v,v_i\}$.

But at most $\gamma n/2$ of these cliques touch a missing edge other than $vv_i$, so there is a $C'$ among them that would have been a good choice for $v$. This contradiction establishes our claim and finishes the proof of the theorem.
\end{proof}

\section{Concluding remarks} \label{sec:last}

Many of the saturation results also generalize to hypergraphs. For example, Bollob\'as \cite{B65} and Alon \cite{A85} proved that $\sat(\Krn,\Krs)=\wsat(\Krn,\Krs)=\binom{n}{r}-\binom{n-s+r}{r}$, where $\Krt$ denotes the complete $r$-uniform hypergraph on $t$ vertices. It would therefore be very interesting to see how saturation behaves in $\Grnp$, the random $r$-uniform hypergraph.

With some extra ideas, our proofs can be adapted to give tight results about $\Krr$-saturated hypergraphs, but the general $\Krs$-saturated case appears to be more difficult. It is possible, however, that our methods can be pushed further to solve these questions, and we plan to return to the problem at a later occasion.

\medskip
Another interesting direction is to study the saturation problem for non-constant probability ranges. For example, Theorem~\ref{thm:main} about strong saturation can be extended to the range $\frac{1}{\log^{\eps(s)} n}\le p\le 1-\frac{1}{o(n)}$ in a fairly straightforward manner. 
With a bit of extra work, the case of $K_3$-saturation can be further extended to $n^{-1/2}\ll p\ll 1$, where one can prove $\sat(G(n,p),K_3)=(1+o(1))\frac{n}{p}\log np^2$ as sketched below.

With $p=o(1)$, the upper bound construction can be simplified: Let $G_0\subs G(n,p)$ contain all edges between a set $A$ of $(1+o(1))\frac{1}{p^2}\log np^2$ vertices and $V-A$. Then $G_0$ will whp complete all but $O(n/p)$ edges of $G(n,p)$, so it can be extended to a $K_3$-saturated subgraph of the desired size. The matching lower bound can be proved along the lines of Theorem~\ref{thm:main_lower}, with a more careful analysis and an extra argument showing that the edges induced by $B$ cannot complete many more than $\frac{n}{p}\log^2 n$ edges.

On the other hand, note that for $p\ll n^{-1/2}$ there are much fewer triangles in $G(n,p)$ than edges, so any saturated graph must contain almost all the edges, i.e., $\sat(G(n,p),K_3)=(1+o(1))\binom{n}{2}p$. This gives us a fairly good understanding of $K_3$-saturation in the random setting.

However, $K_s$-saturation in general appears to be more difficult in sparse random graphs. In particular, when $p$ is much smaller than $\frac{1}{\log n}$, the set $A$ in our construction becomes too sparse to apply Theorem~\ref{thm:kriv} on it. This suggests that here the saturation numbers might depend more on $s$.

\medskip
It is also not difficult to extend our weak saturation result, Theorem~\ref{thm:weaksat}, to the range $n^{-\eps(s)}\le p\le 1$. 
A next step could be to obtain results for smaller values of $p$, in particular, it would be nice to determine the exact probability range 
where the weak saturation number is $(s-2)n-\binom{s-1}{2}$. Note that our lower bound holds as long as $G(n,p)$ is weakly $K_s$-saturated in $K_n$. 
This problem was studied by Balogh, Bollob\'as and Morris \cite{BBM12}, who showed that the probability threshold for $G(n,p)$ to be 
weakly $K_s$-saturated is around $p\approx n^{-1/\lambda(s)}\polylog n$, where $\lambda(s)=\frac{\binom{s}{2}-2}{s-2}$. It might be possible 
that $\wsat(G(n,p),K_s)$ changes its behavior at the same threshold.

\medskip
$F$-saturation in the complete graph has been studied for various different graphs $F$ (see \cite{FFS} for references). For example, K\'aszonyi and Tuza \cite{KT86} showed that for any fixed (connected) graph $F$ on at least 3 vertices, $\sat(n,F)$ is linear in $n$. As we have seen, this is not true in $G(n,p)$. However, analogous results in random host graphs could be of some interest.

\bigskip

\noindent{\bf Acknowledgements.}\, 
We would like to thank Choongbum Lee and Matthew Kwan for stimulating discussions on the topic of this paper.

\end{document}